\documentclass[11pt,a4paper]{article}
\usepackage{amsmath}
\usepackage{amsfonts}
\usepackage{amssymb}

\providecommand{\U}[1]{\protect \rule{.1in}{.1in}}
\newtheorem{theorem}{Theorem}

\newtheorem{corollary}[theorem]{Corollary}

\newtheorem{lemma}[theorem]{Lemma}

\newtheorem{proposition}[theorem]{Proposition}
\newtheorem{remark}[theorem]{Remark}

\newenvironment{proof}[1][Proof]{\noindent \textbf{#1.} }{\  \rule{0.5em}{0.5em}}
\include {mak}
\input{tcilatex}
\begin{document}

\begin{center}
{\LARGE On a class of polynomials connected to Bell polynomials}

\  \  \  \  \ 

MILOUD MIHOUBI\footnote[1]{%
This work is supported by RECITS Laboratory of USTHB.} AND MADJID\ SAHARI%
\footnotemark[2]\medskip \\[0pt]
USTHB, Faculty of Mathematics, BP 32, El-Alia, 16111, Algiers, Algeria. \\[%
0pt]
{\Large {\large \footnotemark[1]}}mmihoubi@usthb.dz \ or {\Large {\large 
\footnotemark[1]}}miloudmihoubi@gmail.com\\[0pt]
{\Large {\large \footnotemark[2]}}msahari@usthb.dz \ or \ {\Large {\large 
\footnotemark[2]}}madjid\_sahari@yahoo.fr

\  \  \  \ 
\end{center}

\noindent \textbf{Abstract. }In this paper, we study a class of sequences of
polynomials linked to the sequence of Bell polynomials. Some sequences of
this class have applications on the theory of hyperbolic differential
equations and other sequences generalize Laguerre polynomials and associated
Lah polynomials. We discuss, for these polynomials, their explicit
expressions, relations to the successive derivatives of a given function,
real zeros and recurrence relations. Some known results are significantly
simplified.



\section{Introduction}

In \cite{erd}, there are many polynomials having applications to the
hyperbolic partial differential equations%
\begin{equation*}
Au_{xx}+2Bu_{xy}+Cu_{yy}+Du_{x}+Eu_{y}+F=0\text{\ with }AD>BC,
\end{equation*}%
for which the following two sequences of polynomials $\left( U_{n}\left(
x\right) \right) $ and $\left( V_{n}\left( x\right) \right) $ defined by%
\begin{eqnarray*}
\underset{n\geq 0}{\sum }U_{n}\left( x\right) \frac{t^{n}}{n!} &=&\left(
1-t\right) ^{-1/2}\exp \left( x\left( \left( 1-t\right) ^{-1/2}-1\right)
\right) , \\
\underset{n\geq 0}{\sum }V_{n}\left( x\right) \frac{t^{n}}{n!} &=&\left(
1-t\right) ^{-3/2}\exp \left( x\left( \left( 1-t\right) ^{-1/2}-1\right)
\right)
\end{eqnarray*}%
are considered. These polynomials have applications to the theory of
hyperbolic partial differential equations, see \cite[pp. 391--398]{cou}.
They can be expressed as follows \cite[pp. 257--258]{erd}:%
\begin{eqnarray*}
U_{n}\left( x\right) &=&xe^{-x}\left( \frac{d}{d\left( x^{2}\right) }\right)
^{n}\left( x^{2n-1}e^{x}\right) , \\
V_{n}\left( x\right) &=&\frac{e^{-x}}{x}\left( \frac{d}{d\left( x^{2}\right) 
}\right) ^{n}\left( x^{2n+1}e^{x}\right) .
\end{eqnarray*}%
Recently, two studies of the sequence of polynomials $\left( U_{n}\left(
x\right) \right) $ are given in \cite{kim,qi}.\newline
Motivated by these works, to give more properties of these polynomials, we
prefer to consider their generalized sequence of polynomials $\left(
P_{n}^{\left( \alpha ,\beta \right) }\left( x\right) \right) $ defined by%
\begin{equation*}
\underset{n\geq 0}{\sum }P_{n}^{\left( \alpha ,\beta \right) }\left(
x\right) \frac{t^{n}}{n!}=\left( 1-t\right) ^{\alpha }\exp \left( x\left(
\left( 1-t\right) ^{\beta }-1\right) \right) ,\  \  \alpha ,\beta \in \mathbb{R%
},\  \beta \neq 0..
\end{equation*}%
The first few values of the sequence $\left( P_{n}^{\left( \alpha ,\beta
\right) }\left( -\frac{x}{\beta }\right) ;n\geq 0\right) $ are to be%
\begin{eqnarray*}
P_{0}^{\left( \alpha ,\beta \right) }\left( -\frac{x}{\beta }\right) &=&1, \\
P_{1}^{\left( \alpha ,\beta \right) }\left( -\frac{x}{\beta }\right)
&=&x-\alpha , \\
P_{2}^{\left( \alpha ,\beta \right) }\left( -\frac{x}{\beta }\right)
&=&x^{2}-\left( 2\alpha +\beta -1\right) x+\left( \alpha \right) _{2}, \\
P_{3}^{\left( \alpha ,\beta \right) }\left( -\frac{x}{\beta }\right)
&=&x^{3}-3\left( \alpha +\beta -1\right) x^{2}+\left( 3\alpha ^{2}+3\alpha
\beta -6\alpha +\beta ^{2}-3\beta +2\right) x-\left( \alpha \right) _{3},
\end{eqnarray*}%
where $\left( \alpha \right) _{n}:=\alpha \left( \alpha -1\right) \cdots
\left( \alpha -n+1\right) $ \ if \ $n\geq 1$ \ and \ $\left( \alpha \right)
_{0}:=1.$\newline
We use also the notation $\left \langle \alpha \right \rangle _{n}:=\alpha
\left( \alpha +1\right) \cdots \left( \alpha +n-1\right) $ \ if \ $n\geq 1$
\ and \ $\left \langle \alpha \right \rangle _{0}:=1.$\newline
The paper is organized as follows. In the next section we give different
expressions for $P_{n}^{\left( \alpha ,\beta \right) }\left( x\right) .$ In
the third section we give special expressions for $P_{n}^{\left( \alpha
,\beta \right) }\left( x\right) $ and we show that it has only real zeros
under certain conditions on $\alpha $\ and $\beta .$ In the fourth section,
we give differential equations arising from of the polynomials $%
P_{n}^{\left( \alpha ,\beta \right) }$ and some recurrence relations, and,
in the last section apply the obtained results to some particular
polynomials.

\section{Explicit expressions for the polynomials $P_{n}^{\left( \protect%
\alpha ,\protect \beta \right) }$}

In this section, we give some explicit expressions for $P_{n}^{\left( \alpha
,\beta \right) }\left( x\right) .$ Two expressions of $P_{n}^{\left( \alpha
,\beta \right) }\left( x\right) $ related to Dobinski's formula and
generalized Stirling numbers are given by the following proposition.

\begin{proposition}
\label{P1}There hold%
\begin{eqnarray*}
P_{n}^{\left( \alpha ,\beta \right) }\left( x\right) &=&e^{-x}\underset{%
k\geq 0}{\sum }\left \langle -\alpha -\beta k\right \rangle _{n}\frac{x^{k}}{%
k!} \\
P_{n}^{\left( \alpha ,\beta \right) }\left( x\right) &=&\underset{k=0}{%
\overset{n}{\sum }}S_{\alpha ,\beta }\left( n,k\right) x^{k},
\end{eqnarray*}%
where 
\begin{equation*}
S_{\alpha ,\beta }\left( n,k\right) =\frac{1}{k!}\underset{j=0}{\overset{k}{%
\sum }}\left( -1\right) ^{k-j}\binom{k}{j}\left \langle -\alpha -\beta
j\right \rangle _{n}.
\end{equation*}
\end{proposition}

\begin{proof}
The first identity follows from the expansion%
\begin{eqnarray*}
\underset{n\geq 0}{\sum }P_{n}^{\left( \alpha ,\beta \right) }\left(
x\right) \frac{t^{n}}{n!} &=&e^{-x}\left( 1-t\right) ^{\alpha }\exp \left(
x\left( 1-t\right) ^{\beta }\right) \\
&=&e^{-x}\underset{k\geq 0}{\sum }x^{k}\frac{\left( 1-t\right) ^{\alpha
+\beta k}}{k!} \\
&=&e^{-x}\underset{k\geq 0}{\sum }\frac{x^{k}}{k!}\underset{n\geq 0}{\sum }%
\left \langle -\alpha -\beta k\right \rangle _{n}\frac{t^{n}}{n!} \\
&=&\underset{n\geq 0}{\sum }\left( e^{-x}\underset{k\geq 0}{\sum }\left
\langle -\alpha -\beta k\right \rangle _{n}\frac{x^{k}}{k!}\right) \frac{%
t^{n}}{n!},
\end{eqnarray*}%
and, the second identity follows from the first by expansion $e^{-x}$ in
power series.
\end{proof}

\begin{proposition}
\label{P5}There holds%
\begin{equation*}
x^{n}=\underset{k=0}{\overset{n}{\sum }}\widetilde{S}_{\alpha ,\beta }\left(
n,k\right) P_{k}^{\left( \alpha ,\beta \right) }\left( x\right) \  \  \text{%
with}\  \  \widetilde{S}_{\alpha ,\beta }\left( n,k\right) =\left( -1\right)
^{n-k}S_{-\frac{\alpha }{\beta },\frac{1}{\beta }}\left( n,k\right) .
\end{equation*}
\end{proposition}

\begin{proof}
Since%
\begin{equation*}
\underset{n\geq k}{\sum }S_{\alpha ,\beta }\left( n,k\right) \frac{t^{n}}{n!}%
=\frac{1}{k!}\underset{n\geq k}{\sum }\left( \frac{d}{dx}\right)
^{k}P_{n}^{\left( \alpha ,\beta \right) }\left( 0\right) \frac{t^{n}}{n!}=%
\frac{1}{k!}\left( \left( 1-t\right) ^{\beta }-1\right) ^{k}\left(
1-t\right) ^{\alpha },
\end{equation*}%
it results%
\begin{eqnarray*}
\underset{n\geq 0}{\sum }\left( \underset{k=0}{\overset{n}{\sum }}\widetilde{%
S}_{\alpha ,\beta }\left( n,k\right) P_{k}^{\left( \alpha ,\beta \right)
}\left( x\right) \right) \frac{t^{n}}{n!} &=&\underset{k\geq 0}{\sum }\left(
-1\right) ^{k}P_{k}^{\left( \alpha ,\beta \right) }\left( x\right) \underset{%
n\geq k}{\sum }S_{-\frac{\alpha }{\beta },\frac{1}{\beta }}\left( n,k\right) 
\frac{\left( -t\right) ^{n}}{n!} \\
&=&\underset{k\geq 0}{\sum }\left( -1\right) ^{k}\frac{P_{k}^{\left( \alpha
,\beta \right) }\left( x\right) }{k!}\left( \left( 1+t\right) ^{\frac{1}{%
\beta }}-1\right) ^{k}\left( 1+t\right) ^{-\frac{\alpha }{\beta }} \\
&=&\left( 1+t\right) ^{-\frac{\alpha }{\beta }}\underset{k\geq 0}{\sum }%
P_{k}^{\left( \alpha ,\beta \right) }\left( x\right) \frac{\left( 1-\left(
1+t\right) ^{\frac{1}{\beta }}\right) ^{k}}{k!} \\
&=&\exp \left( xt\right) ,
\end{eqnarray*}%
which shows the desired identity.
\end{proof}

\begin{corollary}
\label{C4}There holds%
\begin{equation*}
\left \langle -\alpha -\beta x\right \rangle _{n}=\underset{j=0}{\overset{n}{%
\sum }}S_{\alpha ,\beta }\left( n,j\right) \left( x\right) _{j}.
\end{equation*}
\end{corollary}

\begin{proof}
Let $\left \langle -\alpha -\beta x\right \rangle _{n}=\underset{j=0}{%
\overset{n}{\sum }}\delta \left( n,j\right) \left( x\right) _{j}.$ Then,
from Proposition \ref{P1} we get%
\begin{equation*}
P_{n}^{\left( \alpha ,\beta \right) }\left( x\right) =e^{-x}\underset{k\geq 0%
}{\sum }\left \langle -\alpha -\beta k\right \rangle _{n}\frac{x^{k}}{k!}=%
\underset{j=0}{\overset{n}{\sum }}\delta \left( n,j\right) \left( e^{-x}%
\underset{k\geq j}{\sum }\left( k\right) _{j}\frac{x^{k}}{k!}\right) =%
\underset{j=0}{\overset{n}{\sum }}\delta \left( n,j\right) x^{j},
\end{equation*}%
which gives $\delta \left( n,j\right) =S_{\alpha ,\beta }\left( n,j\right) .$
\end{proof}

\noindent If $\mathcal{B}_{n}$ denote the $n$-th Bell polynomial, then when
we replace $t$ by $1-e^{t}$ in the generating function of the sequence $%
\left( P_{n}^{\left( \alpha ,\beta \right) }\left( x\right) \right) ,$ then $%
P_{n}^{\left( \alpha ,\beta \right) }\left( x\right) $ can be written in the
basis $\left \{ 1,\mathcal{B}_{1}\left( x\right) ,\ldots ,\mathcal{B}%
_{n}\left( x\right) \right \} $ as follows:

\begin{proposition}
\label{P2}There holds%
\begin{equation*}
\underset{k=0}{\overset{n}{\sum }}\left( -1\right) ^{k}S\left( n,k\right)
P_{k}^{\left( \alpha ,\beta \right) }\left( x\right) =\underset{k=0}{\overset%
{n}{\sum }}\binom{n}{k}\alpha ^{n-k}\beta ^{k}\mathcal{B}_{k}\left( x\right)
,
\end{equation*}%
or equivalently%
\begin{equation*}
P_{n}^{\left( \alpha ,\beta \right) }\left( x\right) =\underset{j=0}{\overset%
{n}{\sum }}\beta ^{j}\left( \underset{k=j}{\overset{n}{\sum }}\left(
-1\right) ^{k}\left \vert s\left( n,k\right) \right \vert \alpha
^{k-j}\right) \mathcal{B}_{j}\left( x\right) ,
\end{equation*}%
where $s\left( n,k\right) $ and $S\left( n,k\right) $ are, respectively, the
Stirling numbers of the first and second kind, see for instance \cite%
{com,xu1}.
\end{proposition}

\noindent Let $B_{n+r,k+r}^{\left( r\right) }\left( \left( a_{i},i\geq
1\right) ;\left( b_{i},i\geq 1\right) \right) $ are the partial $r$-Bell
polynomials \cite{cho,mih3,sha} defined by%
\begin{equation*}
\underset{n\geq k}{\sum }B_{n+r,k+r}^{\left( r\right) }\left(
a_{l};b_{l}\right) \frac{t^{n}}{n!}=\frac{1}{k!}\left( \underset{j\geq 1}{%
\sum }a_{j}\frac{t^{j}}{j!}\right) ^{k}\left( \underset{j\geq 0}{\sum }%
b_{j+1}\frac{t^{j}}{j!}\right) ^{r}.
\end{equation*}%
and $B_{n,k}\left( \left( a_{i},i\geq 1\right) \right) =B_{n,k}^{\left(
0\right) }\left( \left( a_{i},i\geq 1\right) ;\left( b_{i},i\geq 1\right)
\right) $ are the partial Bell polynomials \cite{abo,com,mih1,mih2}. An
expression of $P_{n}^{\left( \alpha ,\beta \right) }\left( x\right) $ in
terms of the partial $r$-Bell polynomials is as follows.

\begin{proposition}
\label{P3}For any non-negative integer $r,$ there hold%
\begin{equation*}
S_{r\alpha ,\beta }\left( n,k\right) =B_{n+r,k+r}^{\left( r\right) }\left(
\left \langle -\beta \right \rangle _{j};\left \langle -\alpha \right
\rangle _{j-1}\right) ,
\end{equation*}%
which imply%
\begin{equation*}
P_{n}^{\left( r\alpha ,\beta \right) }\left( x\right) =\underset{k=0}{%
\overset{n}{\sum }}B_{n+r,k+r}^{\left( r\right) }\left( \left \langle -\beta
\right \rangle _{j};\left \langle -\alpha \right \rangle _{j-1}\right) x^{k}.
\end{equation*}
\end{proposition}

\begin{proof}
From definition we may state%
\begin{eqnarray*}
\underset{n\geq 0}{\sum }\left( \frac{d}{dx}\right) ^{k}P_{n}^{\left(
r\alpha ,\beta \right) }\left( 0\right) \frac{t^{n}}{n!} &=&\left( \left(
1-t\right) ^{\beta }-1\right) ^{k}\left( 1-t\right) ^{r\alpha } \\
&=&\left( \underset{n\geq 1}{\sum }\left \langle -\beta \right \rangle _{n}%
\frac{t^{n}}{n!}\right) ^{k}\left( \underset{n\geq 0}{\sum }\left \langle
-\alpha \right \rangle _{n}\frac{t^{n}}{n!}\right) ^{r} \\
&=&k!\underset{n\geq k}{\sum }B_{n+r,k+r}^{\left( r\right) }\left( \left
\langle -\beta \right \rangle _{j};\left \langle -\alpha \right \rangle
_{j-1}\right) \frac{t^{n}}{n!},
\end{eqnarray*}%
and by connection to $S_{\alpha ,\beta }\left( n,k\right) $ defined above,
this last expansion gives\newline
$S_{r\alpha ,\beta }\left( n,k\right) =\frac{1}{k!}\left( \frac{d}{dx}%
\right) ^{k}P_{n}^{\left( r\alpha ,\beta \right) }\left( 0\right)
=B_{n+r,k+r}^{\left( r\right) }\left( \left \langle -\beta \right \rangle
_{j};\left \langle -\alpha \right \rangle _{j-1}\right) .$
\end{proof}

\section{Some properties of the polynomials $P_{n}^{\left( \protect \alpha ,%
\protect \beta \right) }$}

In this section, we show the link of the sequence of polynomials $\left(
P_{n}^{\left( \alpha ,\beta \right) }\left( x\right) \right) $ to the
successive derivatives of a given function and we give sufficient conditions
on $\alpha $ and $\beta $ for which the polynomial $P_{n}^{\left( \alpha
,\beta \right) }$ has only real zeros.

\begin{lemma}
\label{L1}There holds%
\begin{equation*}
P_{n+1}^{\left( \alpha ,\beta \right) }\left( x\right) =\left( n-\alpha
-\beta x\right) P_{n}^{\left( \alpha ,\beta \right) }\left( x\right) -\beta x%
\frac{d}{dx}P_{n}^{\left( \alpha ,\beta \right) }\left( x\right) .
\end{equation*}
\end{lemma}

\begin{proof}
One can verify that the function%
\begin{equation*}
F_{\alpha ,\beta }\left( t,x\right) =\left( 1-t\right) ^{\alpha }\exp \left(
x\left( \left( 1-t\right) ^{\beta }-1\right) \right)
\end{equation*}%
is a solution of the partial differential equation%
\begin{equation*}
\left( 1-t\right) \frac{d}{dt}Y+\beta x\frac{d}{dx}Y+\left( \alpha +\beta
x\right) Y=0
\end{equation*}%
from which it results the desired identity.
\end{proof}

\begin{theorem}
\label{T4}For $x>0$ and $n\geq 0$ we have%
\begin{eqnarray*}
P_{n}^{\left( \alpha ,\beta \right) }\left( x^{\beta }\right) &=&\left(
-1\right) ^{n}x^{n-\alpha }e^{-x^{\beta }}\left( \frac{d}{dx}\right)
^{n}\left( x^{\alpha }e^{x^{\beta }}\right) , \\
P_{n}^{\left( \alpha ,\beta \right) }\left( x^{-\beta }\right) &=&x^{\alpha
+1}e^{-x^{-\beta }}\left( \frac{d}{dx}\right) ^{n}\left( x^{n-1-\alpha
}e^{x^{-\beta }}\right) .
\end{eqnarray*}
\end{theorem}

\begin{proof}
For the first identity, Lemma \ref{L1} gives%
\begin{equation*}
P_{n}^{\left( \alpha ,\beta \right) }\left( x\right) =\left( n-1-\alpha
-\beta x\right) P_{n-1}^{\left( \alpha ,\beta \right) }\left( x\right)
-\beta x\frac{d}{dx}P_{n-1}^{\left( \alpha ,\beta \right) }\left( x\right) ,
\end{equation*}%
and if we set $f_{n}^{\left( \alpha ,\beta \right) }\left( x\right) :=\left(
-1\right) ^{n}x^{\frac{\alpha -n}{\beta }}e^{x}P_{n}^{\left( \alpha ,\beta
\right) }\left( x\right) ,$ the last identity can also be written as%
\begin{equation*}
f_{n}^{\left( \alpha ,\beta \right) }\left( x\right) =\beta x^{1-\frac{1}{%
\beta }}\frac{d}{dx}f_{n-1}^{\left( \alpha ,\beta \right) }\left( x\right) =%
\frac{d}{d\left( x^{1/\beta }\right) }f_{n-1}^{\left( \alpha ,\beta \right)
}\left( x\right)
\end{equation*}%
which implies $f_{n}^{\left( \alpha ,\beta \right) }\left( x\right) =\left( 
\frac{d}{d\left( x^{1/\beta }\right) }\right) ^{n}f_{0}^{\left( \alpha
,\beta \right) }\left( x\right) =\left( \frac{d}{d\left( x^{1/\beta }\right) 
}\right) ^{n}\left( x^{\frac{\alpha }{\beta }}e^{x}\right) .$ So, we get%
\begin{equation*}
P_{n}^{\left( \alpha ,\beta \right) }\left( x\right) =\left( -1\right)
^{n}x^{\frac{n-\alpha }{\beta }}e^{-x}\left( \frac{d}{d\left( x^{1/\beta
}\right) }\right) ^{n}\left( x^{\frac{\alpha }{\beta }}e^{x}\right)
\end{equation*}%
or equivalently $P_{n}^{\left( \alpha ,\beta \right) }\left( y^{\beta
}\right) =\left( -1\right) ^{n}y^{n-\alpha }e^{-y^{\beta }}\left( \frac{d}{dy%
}\right) ^{n}\left( y^{\alpha }e^{y^{\beta }}\right) .$\newline
For the second identity, we proceed as follows%
\begin{eqnarray*}
\left( \frac{d}{d\left( x^{-1/\beta }\right) }\right) ^{n}\left( x^{\frac{%
\alpha +1-n}{\beta }}e^{x}\right) &=&\left( \frac{d}{dy}\right)
_{y=x^{-1/\beta }}^{n}\left( y^{n-\alpha -1}e^{y^{-\beta }}\right) \\
&=&\left( \frac{d}{dy}\right) _{y=x^{-1/\beta }}^{n}\underset{k\geq 0}{\sum }%
\frac{1}{k!}y^{n-1-\alpha -k\beta } \\
&=&\underset{k\geq 0}{\sum }\left( n-1-\alpha -k\beta \right) _{n}\left. 
\frac{y^{-1-\alpha -k\beta }}{k!}\right \vert _{y=x^{-1/\beta }} \\
&=&x^{\frac{\alpha +1}{\beta }}\underset{k\geq 0}{\sum }\left \langle
-\alpha -k\beta \right \rangle _{n}\frac{x^{k}}{k!} \\
&=&x^{\frac{\alpha +1}{\beta }}e^{x}P_{n}^{\left( \alpha ,\beta \right)
}\left( x\right) ,
\end{eqnarray*}%
i.e. $P_{n}^{\left( \alpha ,\beta \right) }\left( x\right) =x^{-\frac{\alpha
+1}{\beta }}e^{-x}\left( \frac{d}{d\left( x^{-1/\beta }\right) }\right)
^{n}\left( x^{\frac{\alpha +1-n}{\beta }}e^{x}\right) $ which is equivalent
to the desired identity.
\end{proof}

\begin{remark}
The first hand of Theorem \ref{T4} can also be proved directly using
Proposition \ref{P1} as follows:%
\begin{eqnarray*}
P_{n}^{\left( \alpha ,\beta \right) }\left( x^{\beta }\right)
&=&e^{-x^{\beta }}\underset{k\geq 0}{\sum }\left \langle -\alpha -\beta
k\right \rangle _{n}\frac{x^{\beta k}}{k!} \\
&=&\left( -1\right) ^{n}x^{n-\alpha }e^{-x^{\beta }}\underset{k\geq 0}{\sum }%
\left( \alpha +\beta k\right) _{n}\frac{x^{\alpha +\beta k-n}}{k!} \\
&=&\left( -1\right) ^{n}x^{n-\alpha }e^{-x^{\beta }}\left( \frac{d}{dx}%
\right) ^{n}\left( \underset{k\geq 0}{\sum }\frac{x^{\alpha +\beta k}}{k!}%
\right) \\
&=&\left( -1\right) ^{n}x^{n-\alpha }e^{-x^{\beta }}\left( \frac{d}{dx}%
\right) ^{n}\left( x^{\alpha }e^{x^{\beta }}\right) .
\end{eqnarray*}
\end{remark}

\noindent Theorem \ref{T4} proves that $\mathcal{B}_{n}\left( x^{\beta
}\right) $ can also be written in a similar form as follows.

\begin{proposition}
For $x>0$ and $n\geq 0$ we have%
\begin{equation*}
\mathcal{B}_{n}\left( x^{\beta }\right) =x^{-\alpha }e^{-x^{\beta }}\left( x%
\frac{d}{dx}-\frac{\alpha }{\beta }\right) ^{n}\left( x^{\alpha }e^{x^{\beta
}}\right) ,\  \  \beta \neq 0.
\end{equation*}
\end{proposition}

\begin{proof}
From Proposition \ref{P2} we get%
\begin{equation*}
\frac{1}{\alpha ^{n}}\underset{k=0}{\overset{n}{\sum }}\left( -1\right)
^{k}S\left( n,k\right) P_{k}^{\left( \alpha ,\beta \right) }\left( x\right) =%
\underset{k=0}{\overset{n}{\sum }}\binom{n}{k}\left( \frac{\beta }{\alpha }%
\right) ^{k}\mathcal{B}_{k}\left( x\right) ,
\end{equation*}%
which implies by using Theorem \ref{T4}%
\begin{eqnarray*}
\left( \frac{\beta }{\alpha }\right) ^{n}\mathcal{B}_{n}\left( x^{\beta
}\right)  &=&\underset{k=0}{\overset{n}{\sum }}\left( -1\right) ^{n-k}\binom{%
n}{k}\frac{1}{\alpha ^{k}}\underset{j=0}{\overset{k}{\sum }}\left( -1\right)
^{j}S\left( k,j\right) P_{j}^{\left( \alpha ,\beta \right) }\left( x^{\beta
}\right)  \\
&=&\underset{k=0}{\overset{n}{\sum }}\left( -1\right) ^{n-k}\binom{n}{k}%
\frac{1}{\alpha ^{k}}\underset{j=0}{\overset{k}{\sum }}\left( -1\right)
^{j}S\left( k,j\right) \left( -1\right) ^{j}x^{j-\alpha }e^{-x^{\beta
}}\left( \frac{d}{dx}\right) ^{j}\left( x^{\alpha }e^{x^{\beta }}\right)  \\
&=&\left( -1\right) ^{n}x^{-\alpha }e^{-x^{\beta }}\underset{k=0}{\overset{n}%
{\sum }}\binom{n}{k}\left( -\frac{1}{\alpha }\right) ^{k}\underset{j=0}{%
\overset{k}{\sum }}S\left( k,j\right) x^{j}\left( \frac{d}{dx}\right)
^{j}\left( x^{\alpha }e^{x^{\beta }}\right) .
\end{eqnarray*}%
On using the identity $\underset{j=0}{\overset{k}{\sum }}S\left( k,j\right)
x^{j}\left( \frac{d}{dx}\right) ^{j}=\left( x\frac{d}{dx}\right) ^{k}$ \cite%
{eu}, it follows 
\begin{eqnarray*}
\mathcal{B}_{n}\left( x^{\beta }\right)  &=&\left( -\frac{\alpha }{\beta }%
\right) ^{n}x^{-\alpha }e^{-x^{\beta }}\underset{k=0}{\overset{n}{\sum }}%
\binom{n}{k}\left( -\frac{x}{\alpha }\frac{d}{dx}\right) ^{k}\left(
x^{\alpha }e^{x^{\beta }}\right)  \\
&=&\left( -\frac{\alpha }{\beta }\right) ^{n}x^{-\alpha }e^{-x^{\beta
}}\left( 1-\frac{x}{\alpha }\frac{d}{dx}\right) ^{n}\left( x^{\alpha
}e^{x^{\beta }}\right)  \\
&=&x^{-\alpha }e^{-x^{\beta }}\left( x\frac{d}{dx}-\frac{\alpha }{\beta }%
\right) ^{n}\left( x^{\alpha }e^{x^{\beta }}\right) ,
\end{eqnarray*}%
which remains true for $\alpha =0.$
\end{proof}

\noindent To study the real zeros of $P_{n}^{\left( \alpha ,\beta \right) },$
we use the following known theorem. Indeed, let $P_{1}$ and $P_{2}$ be two
polynomials having only real zeros and let $x_{n}\leq \cdots \leq x_{1}$ and 
$y_{m}\leq \cdots \leq y_{1}$ be the zeros of $P_{1}$ and $P_{2},$
respectively. Following \cite{wag}, we say that $P_{2}$ interlaces $P_{1}$
if $m=n-1$ and%
\begin{equation*}
x_{n}\leq y_{n-1}\leq x_{n-1}\leq \cdots \leq y_{1}\leq x_{1}
\end{equation*}%
and that $P_{2}$ alternates left of $P_{1}$ if $m=n$ and%
\begin{equation*}
y_{n}\leq x_{n}\leq y_{n-1}\leq x_{n-1}\leq \cdots \leq y_{1}\leq x_{1}.
\end{equation*}

\begin{theorem}
\label{T1}\cite[Th. 1]{wan} Let $a_{1},a_{2},\ b_{1},\ b_{2}$ be real
numbers, let $P_{1},$ $P_{2}$ be two polynomials whose leading coefficients
have the same sign and let $P\left( x\right) =\left( a_{1}x+b_{1}\right)
P_{1}\left( x\right) +\left( a_{2}x+b_{2}\right) P_{2}\left( x\right) .$
Suppose that $P_{1},\ P_{2}$ have only real zeros and $P_{2}$ interlaces $%
P_{1}$ or $P_{2}$ alternates left of $P_{1}.$ Then, if $a_{1}b_{2}\geq
b_{1}a_{2},$ $P(x)\ $has only real zeros.
\end{theorem}

\begin{theorem}
\label{T3}Let%
\begin{eqnarray*}
A &=&\left \{ \left( \alpha ,\beta \right) \in \mathbb{R}^{2}:\left( \beta
-1\right) ^{2}+4\alpha \beta \geq 0,\  \  \beta <0,\  \alpha \leq 2\right \} ,
\\
\widetilde{A} &=&\left \{ \left( \alpha ,\beta \right) \in \mathbb{R}%
^{2}:\beta >0,\  \alpha \geq 1\right \} .
\end{eqnarray*}%
Then, for $\left( \alpha ,\beta \right) \in A,$ the polynomial $%
P_{n}^{\left( \alpha ,\beta \right) }$ has only real zeros, $n\geq 1,$ and,
for $\left( \alpha ,\beta \right) \in \widetilde{A},$ the polynomials $%
P_{1}^{\left( \alpha ,\beta \right) },\ldots ,P_{\left \lceil \alpha
\right
\rceil }^{\left( \alpha ,\beta \right) }$ has only real zeros, where 
$\left
\lceil \alpha \right \rceil $ is the smallest integer $\geq \alpha .$
\end{theorem}

\begin{proof}
We proceed by induction on $n\geq 1.$ For $n=1,$ the polynomial $%
P_{1}^{\left( \alpha \right) }\left( x\right) =-\beta x-\alpha $ has a real
zero, and for $n=2,$ the polynomial 
\begin{equation*}
P_{2}^{\left( \alpha \right) }\left( x\right) =\beta ^{2}x^{2}+\beta \left(
2\alpha +\beta -1\right) x+\alpha \left( \alpha -1\right)
\end{equation*}%
has only real zeros when $\left( \beta -1\right) ^{2}+4\alpha \beta \geq 0$
and $\beta <0.$\newline
Assume that $P_{n}^{\left( \alpha ,\beta \right) }\left( x\right) $ has $n$ $%
\left( \geq 2\right) $ real zeros different from zero, since the heading
coefficient of $P_{n}^{\left( \alpha ,\beta \right) }\left( x\right) $ is $%
S_{\alpha ,\beta }\left( n,n\right) =\left( -\beta \right) ^{n}$ and then
heading coefficient of $\frac{d}{dx}P_{n}^{\left( \alpha ,\beta \right)
}\left( x\right) $ is $nS_{\alpha ,\beta }\left( n,n\right) =n\left( -\beta
\right) ^{n},$ then they are of the same sign. Also, since $\frac{d}{dx}%
P_{n}^{\left( \alpha ,\beta \right) }\left( x\right) $ interlaces $%
P_{n}^{\left( \alpha ,\beta \right) }\left( x\right) $ it follows from
Theorem \ref{T1} that if $-\beta \left( n-\alpha \right) \geq 0,$ $%
P_{n+1}^{\left( \alpha ,\beta \right) }$ has only real zeros. The condition $%
-\beta \left( n-\alpha \right) \geq 0$ is satisfied when $\left( \alpha
,\beta \right) \in A$ because $-\beta \left( n-\alpha \right) \geq -\beta
\left( 2-\alpha \right) \geq 0.$ It is also satisfied when $n\in \left[
1,\left \lceil \alpha \right \rceil -1\right] $ and $\left( \alpha ,\beta
\right) \in \widetilde{A}$ because $-\beta \left( n-\alpha \right) \geq
-\beta \left( \left \lceil \alpha \right \rceil -1-\alpha \right) \geq 0.$
\end{proof}

\begin{corollary}
\label{C1}For $\alpha \leq 0$ and $\beta <0$ the sequence $\left( S_{\alpha
,\beta }\left( n,k\right) ;\ 0\leq k\leq n\right) $ is strictly log-concave,
more precisely%
\begin{equation*}
\left( S_{\alpha ,\beta }\left( n,k\right) \right) ^{2}\geq \left( 1+\frac{1%
}{k}\right) \left( 1+\frac{1}{n-k}\right) S_{\alpha ,\beta }\left(
n,k+1\right) S_{\alpha ,\beta }\left( n,k-1\right) ,\  \ 1\leq k\leq n-1.
\end{equation*}
\end{corollary}

\begin{proof}
For $\alpha \leq 0$ and $\beta <0$ the polynomial $P_{n}^{\left( \alpha
,\beta \right) }\left( x\right) $ has only real zeros and its coefficients $%
S_{\alpha ,\beta }\left( n,k\right) $ are non-negative, so Newton's
inequality \cite[pp. 52]{har} completes the proof.
\end{proof}

\section{Differential equations and recurrence relations}

In \cite[Sec. 2]{kim} (see also \cite{qi}), the authors give a differential
equation having as solution the function 
\begin{equation*}
\left( 1-t\right) ^{-1/2}\exp \left( x\left( \left( 1-t\right)
^{-1/2}-1\right) \right)
\end{equation*}%
from which they conclude a generalized recurrence relation for the sequence $%
\left( U_{n}\left( x\right) \right) .$ The results of this section simplify
and generalize these results.

\begin{theorem}
\label{T2}Let $m$ be a positive integer. he function 
\begin{equation*}
F_{\alpha ,\beta }\left( t,x\right) :=\left( 1-t\right) ^{\alpha }\exp
\left( x\left( \left( 1-t\right) ^{\beta }-1\right) \right)
\end{equation*}%
satisfies%
\begin{equation*}
\left( \frac{d}{dt}\right) ^{m}F_{\alpha ,\beta }\left( t,x\right)
=F_{\alpha ,\beta }\left( t,x\right) \left( 1-t\right) ^{-m}P_{m}^{\left(
\alpha ,\beta \right) }\left( x\left( 1-t\right) ^{\beta }\right) .
\end{equation*}
\end{theorem}

\begin{proof}
From the definition of $F_{\alpha ,\beta }\left( t,x\right) $ and Corollary %
\ref{C4} we obtain%
\begin{eqnarray*}
\left( \frac{d}{dt}\right) ^{m}F_{\alpha ,\beta }\left( t,x\right) &=&\left( 
\frac{d}{dt}\right) ^{m}\left( \left( 1-t\right) ^{\alpha }\exp \left(
x\left( \left( 1-t\right) ^{\beta }-1\right) \right) \right) \\
&=&e^{-x}\underset{k\geq 0}{\sum }\frac{x^{k}}{k!}\left( \frac{d}{dt}\right)
^{m}\left( 1-t\right) ^{\alpha +k\beta } \\
&=&e^{-x}\underset{k\geq 0}{\sum }\frac{x^{k}}{k!}\left \langle -\alpha
-k\beta \right \rangle _{m}\left( 1-t\right) ^{\alpha +k\beta -m} \\
&=&e^{-x}\left( 1-t\right) ^{\alpha -m}\underset{j=0}{\overset{m}{\sum }}%
S_{\alpha ,\beta }\left( m,j\right) \underset{k\geq 0}{\sum }\left( k\right)
_{j}\frac{x^{k}\left( 1-t\right) ^{k\beta }}{k!} \\
&=&e^{-x}\left( 1-t\right) ^{\alpha -m}\underset{j=0}{\overset{m}{\sum }}%
S_{\alpha ,\beta }\left( m,j\right) x^{j}\left( 1-t\right) ^{j\beta }%
\underset{k\geq 0}{\sum }\frac{x^{k}\left( 1-t\right) ^{k\beta }}{k!} \\
&=&\left( 1-t\right) ^{\alpha -m}\exp \left( x\left( \left( 1-t\right)
^{\beta }-1\right) \right) \underset{j=0}{\overset{n}{\sum }}S_{\alpha
,\beta }\left( m,j\right) x^{j}\left( 1-t\right) ^{j\beta } \\
&=&F_{\alpha ,\beta }\left( t,x\right) \left( 1-t\right) ^{-m}\underset{j=0}{%
\overset{m}{\sum }}S_{\alpha ,\beta }\left( m,j\right) x^{j}\left(
1-t\right) ^{j\beta } \\
&=&F_{\alpha ,\beta }\left( t,x\right) \left( 1-t\right) ^{-m}P_{m}^{\left(
\alpha ,\beta \right) }\left( x\left( 1-t\right) ^{\beta }\right) .
\end{eqnarray*}
\end{proof}

\noindent The next corollary gives an expression of $P_{n+m}^{\left( \alpha
,\beta \right) }\left( x\right) $ in terms of the family $\left(
x^{k}P_{j}^{\left( \alpha ,\beta \right) }\left( x\right) \right) .$ The
obtained expression is similar to the expression of the Bell number $B_{n+m}$
given in \cite{spi}, Bell polynomial $B_{n+m}$ given in \cite{bel,gou} and
several generalizations given later, see \cite{kar,mah,mez,mih4,xu}.

\begin{corollary}
\label{C3}For $n,m=0,1,2,\ldots ,$ we have%
\begin{equation*}
P_{n+m}^{\left( \alpha ,\beta \right) }\left( x\right) =\underset{j=0}{%
\overset{n}{\sum }}\underset{k=0}{\overset{m}{\sum }}\binom{n}{j}\left
\langle m-\beta k\right \rangle _{n-j}S_{\alpha ,\beta }\left( m,k\right)
x^{k}P_{j}^{\left( \alpha ,\beta \right) }\left( x\right) .
\end{equation*}%
In particular, for $m=1,$ we obtain%
\begin{equation*}
P_{n+1}^{\left( \alpha ,\beta \right) }\left( x\right) =-\underset{j=0}{%
\overset{n}{\sum }}\binom{n}{j}\left( \alpha \left( n-j\right) !+\beta
x\left \langle 1-\beta \right \rangle _{n-j}\right) P_{j}^{\left( \alpha
,\beta \right) }\left( x\right) .
\end{equation*}
\end{corollary}

\begin{proof}
On using Theorem \ref{T2}, we have%
\begin{eqnarray*}
\underset{n\geq 0}{\sum }P_{n+m}^{\left( \alpha ,\beta \right) }\left(
x\right) \frac{t^{n}}{n!} &=&\left( \frac{d}{dt}\right) ^{m}F_{\alpha ,\beta
}\left( t,x\right) \\
&=&F_{\alpha ,\beta }\left( t,x\right) \left( 1-t\right) ^{-m}P_{m}^{\left(
\alpha ,\beta \right) }\left( x\left( 1-t\right) ^{\beta }\right) \\
&=&\left( \underset{i\geq 0}{\sum }P_{i}^{\left( \alpha ,\beta \right)
}\left( x\right) \frac{t^{i}}{i!}\right) \left( \underset{k=0}{\overset{m}{%
\sum }}S_{\alpha ,\beta }\left( m,k\right) x^{k}\left( 1-t\right) ^{-m+\beta
k}\right)
\end{eqnarray*}%
\begin{eqnarray*}
&=&\underset{k=0}{\overset{m}{\sum }}S_{\alpha ,\beta }\left( m,k\right)
x^{k}\left( \underset{i\geq 0}{\sum }P_{i}^{\left( \alpha ,\beta \right)
}\left( x\right) \frac{t^{i}}{i!}\right) \left( \underset{j\geq 0}{\sum }%
\left \langle m-\beta k\right \rangle _{j}\frac{t^{j}}{j!}\right) \\
&=&\underset{k=0}{\overset{m}{\sum }}S_{\alpha ,\beta }\left( m,k\right)
x^{k}\underset{n\geq 0}{\sum }\left( \underset{j=0}{\overset{n}{\sum }}%
\binom{n}{j}\left \langle m-\beta k\right \rangle _{n-j}P_{j}^{\left( \alpha
,\beta \right) }\left( x\right) \right) \frac{t^{n}}{n!} \\
&=&\underset{n\geq 0}{\sum }\left( \underset{j=0}{\overset{n}{\sum }}%
\underset{k=0}{\overset{m}{\sum }}\binom{n}{j}\left \langle m-\beta k\right
\rangle _{n-j}S_{\alpha ,\beta }\left( m,k\right) x^{k}P_{j}^{\left( \alpha
,\beta \right) }\left( x\right) \right) \frac{t^{n}}{n!}
\end{eqnarray*}%
which follows gives the desired identity.
\end{proof}

\begin{remark}
For $n=1$ in Corollary \ref{C3} we get%
\begin{equation*}
P_{m+1}^{\left( \alpha ,\beta \right) }\left( x\right) =\underset{j=0}{%
\overset{m+1}{\sum }}\left( \left( m-\alpha -\beta j\right) S_{\alpha ,\beta
}\left( m,j\right) -\beta S_{\alpha ,\beta }\left( m,j-1\right) \right)
x^{j}.
\end{equation*}%
So, since from Proposition \ref{P1} we have $P_{m+1}^{\left( \alpha ,\beta
\right) }\left( x\right) =\underset{j=0}{\overset{m+1}{\sum }}S_{\alpha
,\beta }\left( m+1,j\right) x^{j},$ it results%
\begin{equation*}
S_{\alpha ,\beta }\left( m+1,j\right) =\left( m-\alpha -\beta j\right)
S_{\alpha ,\beta }\left( m,j\right) -\beta S_{\alpha ,\beta }\left(
m,j-1\right) ,
\end{equation*}%
with $S_{\alpha ,\beta }\left( m+1,j\right) =0$ if $j<0$ or $j>m+1.$
\end{remark}

\begin{proposition}
\label{P4}There holds%
\begin{equation*}
P_{n}^{\left( \alpha ,\beta \right) }\left( x\right) =\underset{k=0}{\overset%
{n}{\sum }}\left( -1\right) ^{j}S_{\alpha -\frac{\alpha ^{\prime }}{\beta
^{\prime }}\beta ,\frac{\beta }{\beta ^{\prime }}}\left( n,k\right)
P_{k}^{\left( \alpha ^{\prime },\beta ^{\prime }\right) }\left( x\right) .
\end{equation*}%
In particular, for $\left( \alpha ^{\prime },\beta ^{\prime }\right) =\left(
\alpha ,\beta \right) ,$ we get%
\begin{equation*}
P_{n}^{\left( \lambda \alpha ,\lambda \beta \right) }\left( x\right) =%
\underset{k=0}{\overset{n}{\sum }}\left( -1\right) ^{k}B_{n,k}\left( \left
\langle -\lambda \right \rangle _{j}\right) P_{k}^{\left( \alpha ,\beta
\right) }\left( x\right) ,\  \  \lambda \in \mathbb{R},
\end{equation*}
\end{proposition}

\begin{proof}
From Proposition \ref{P5} we have $x^{k}=\underset{j=0}{\overset{k}{\sum }}%
\left( -1\right) ^{k-j}S_{-\frac{\alpha ^{\prime }}{\beta ^{\prime }},\frac{1%
}{\beta ^{\prime }}}\left( k,j\right) P_{j}^{\left( \alpha ^{\prime },\beta
^{\prime }\right) }\left( x\right) .$ \newline
So, use the identity $P_{n}^{\left( \alpha ,\beta \right) }\left( x\right) =%
\underset{k=0}{\overset{n}{\sum }}S_{\alpha ,\beta }\left( n,k\right) x^{k}$
of Proposition \ref{P1} to obtain

\begin{eqnarray*}
P_{n}^{\left( \alpha ,\beta \right) }\left( x\right) &=&\underset{k=0}{%
\overset{n}{\sum }}S_{\alpha ,\beta }\left( n,k\right) \left( \underset{j=0}{%
\overset{k}{\sum }}\left( -1\right) ^{k-j}S_{-\frac{\alpha ^{\prime }}{\beta
^{\prime }},\frac{1}{\beta ^{\prime }}}\left( k,j\right) P_{j}^{\left(
\alpha ^{\prime },\beta ^{\prime }\right) }\left( x\right) \right) \\
&=&\underset{j=0}{\overset{n}{\sum }}\left( \underset{k=j}{\overset{n}{\sum }%
}\left( -1\right) ^{k-j}S_{\alpha ,\beta }\left( n,k\right) S_{-\frac{\alpha
^{\prime }}{\beta ^{\prime }},\frac{1}{\beta ^{\prime }}}\left( k,j\right)
\right) P_{j}^{\left( \alpha ^{\prime },\beta ^{\prime }\right) }\left(
x\right) .
\end{eqnarray*}%
Now, since from the proof of Proposition \ref{P5}, we have%
\begin{equation*}
\underset{n\geq k}{\sum }S_{\alpha ,\beta }\left( n,k\right) \frac{t^{n}}{n!}%
=\frac{1}{k!}\left( \left( 1-t\right) ^{\beta }-1\right) ^{k}\left(
1-t\right) ^{\alpha },
\end{equation*}%
it follows that the exponential generating function of the sequence $\left(
M\left( n,j\right) ;n\geq j\right) $ defined by%
\begin{equation*}
M\left( n,j\right) :=\underset{k=j}{\overset{n}{\sum }}\left( -1\right)
^{k-j}S_{\alpha ,\beta }\left( n,k\right) S_{-\frac{\alpha ^{\prime }}{\beta
^{\prime }},\frac{1}{\beta ^{\prime }}}\left( k,j\right)
\end{equation*}%
is to be%
\begin{equation*}
\underset{n\geq j}{\sum }M\left( n,j\right) \frac{t^{n}}{n!}=\frac{\left(
-1\right) ^{j}}{j!}\left( 1-t\right) ^{\alpha -\frac{\alpha ^{\prime }}{%
\beta ^{\prime }}\beta }\left( \left( 1-t\right) ^{\frac{\beta }{\beta
^{\prime }}}-1\right) ^{j}
\end{equation*}%
which shows that $M\left( n,j\right) =\left( -1\right) ^{j}S_{\alpha -\frac{%
\alpha ^{\prime }}{\beta ^{\prime }}\beta ,\frac{\beta }{\beta ^{\prime }}%
}\left( n,j\right) .$
\end{proof}

\noindent As a consequence of Proposition \ref{P4}, by combining it with
Propositions \ref{P1}, it results:

\begin{corollary}
\label{C5}For any real numbers $\alpha ,\alpha ^{\prime },\beta ,\beta
^{\prime }$ such that $\beta ^{\prime }\neq 0,$ there hold%
\begin{eqnarray*}
\left \langle -\alpha -\beta x\right \rangle _{n} &=&\underset{j=0}{\overset{%
n}{\sum }}\left( -1\right) ^{j}S_{\alpha -\frac{\alpha ^{\prime }}{\beta
^{\prime }}\beta ,\frac{\beta }{\beta ^{\prime }}}\left( n,j\right) \left
\langle -\alpha ^{\prime }-\beta ^{\prime }x\right \rangle _{j}, \\
S_{\alpha ,\beta }\left( n,k\right) &=&\underset{j=k}{\overset{n}{\sum }}%
\left( -1\right) ^{j}S_{\alpha -\frac{\alpha ^{\prime }}{\beta ^{\prime }}%
\beta ,\frac{\beta }{\beta ^{\prime }}}\left( n,j\right) S_{\alpha ^{\prime
},\beta ^{\prime }}\left( j,k\right) .
\end{eqnarray*}
\end{corollary}

\section{Application to particular polynomials}

\subsection{Application to the polynomials $U_{n}$ and $V_{n}$}

For $n\geq 1,$ the polynomials $U_{n}=P_{n}^{\left( -1/2,-1/2\right) }$ and $%
V_{n}=P_{n}^{\left( -3/2,-1/2\right) }$ defined above, Propositions \ref{P1}%
, \ref{P2} and \ref{P3} give%
\begin{equation*}
\begin{array}{ll}
U_{n}\left( x\right) =e^{-x}\underset{k\geq 0}{\sum }\left \langle \frac{k+1%
}{2}\right \rangle _{n}\frac{x^{k}}{k!}, & V_{n}\left( x\right) =e^{-x}%
\underset{k\geq 0}{\sum }\left \langle \frac{k+3}{2}\right \rangle _{n}\frac{%
x^{k}}{k!}, \\ 
&  \\ 
U_{n}\left( x\right) =\underset{k=0}{\overset{n}{\sum }}S_{-1/2,-1/2}\left(
n,k\right) x^{k}, & V_{n}\left( x\right) =\underset{k=0}{\overset{n}{\sum }}%
S_{-3/2,-1/2}\left( n,k\right) x^{k}, \\ 
&  \\ 
U_{n}\left( x\right) =\underset{j=0}{\overset{n}{\sum }}\left( \underset{k=j}%
{\overset{n}{\sum }}\frac{\left \vert s\left( n,k\right) \right \vert }{2^{k}%
}\right) \mathcal{B}_{j}\left( x\right) , & V_{n}\left( x\right) =\underset{%
j=0}{\overset{n}{\sum }}\frac{1}{3^{j}}\left( \underset{k=j}{\overset{n}{%
\sum }}\left \vert s\left( n,k\right) \right \vert \left( \frac{3}{2}\right)
^{k}\right) \mathcal{B}_{j}\left( x\right) , \\ 
&  \\ 
U_{n}\left( x\right) =\underset{k=0}{\overset{n}{\sum }}B_{n+1,k+1}^{\left(
1\right) }\left( \left \langle \frac{1}{2}\right \rangle _{j};\left \langle 
\frac{1}{2}\right \rangle _{j-1}\right) x^{k}, & V_{n}\left( x\right) =%
\underset{k=0}{\overset{n}{\sum }}B_{n+1,k+1}^{\left( 1\right) }\left( \left
\langle \frac{1}{2}\right \rangle _{j};\left \langle \frac{3}{2}\right
\rangle _{j-1}\right) x^{k}.%
\end{array}%
\end{equation*}%
Theorem \ref{T3} proves that the polynomials $U_{n}$ and $V_{n},$ $n\geq 1,$
have only real zeros and Theorem \ref{T4} shows that, for $x>0,$ there hold%
\begin{eqnarray*}
U_{n}\left( \frac{1}{\sqrt{x}}\right) &=&\left( -1\right) ^{n}x^{n}\sqrt{x}%
e^{-\frac{1}{\sqrt{x}}}\left( \frac{d}{dx}\right) ^{n}\left( \frac{1}{\sqrt{x%
}}e^{\frac{1}{\sqrt{x}}}\right) , \\
V_{n}\left( \frac{1}{\sqrt{x}}\right) &=&\left( -1\right) ^{n}x^{n+1}\sqrt{x}%
e^{-\frac{1}{\sqrt{x}}}\left( \frac{d}{dx}\right) ^{n}\left( \frac{1}{x\sqrt{%
x}}e^{\frac{1}{\sqrt{x}}}\right)
\end{eqnarray*}%
and%
\begin{eqnarray*}
U_{n}\left( \sqrt{x}\right) &=&\sqrt{x}e^{-\sqrt{x}}\left( \frac{d}{dx}%
\right) ^{n}\left( x^{n-1}\sqrt{x}e^{\sqrt{x}}\right) , \\
V_{n}\left( \sqrt{x}\right) &=&\sqrt{x}e^{-x\sqrt{x}}\left( \frac{d}{dx}%
\right) ^{n}\left( x^{n-1}\sqrt{x}e^{x\sqrt{x}}\right) .
\end{eqnarray*}

\subsection{Application to the generalized Laguerre polynomials}

We note here that the sequence of generalized Laguerre polynomials $\left(
L_{n}^{\left( \lambda \right) }\left( x\right) \right) $ (see for example 
\cite{boy,djo,ras}) defined by%
\begin{equation*}
\underset{n\geq 0}{\sum }L_{n}^{\left( \lambda \right) }\left( x\right)
t^{n}=\left( 1-t\right) ^{-\lambda -1}\exp \left( -\frac{xt}{1-t}\right)
\end{equation*}%
presents a particular case of the sequence $\left( P_{n}^{\left( \alpha
,\beta \right) }\left( x\right) \right) ,$ i.e. $L_{n}^{\left( \lambda
\right) }\left( x\right) =\frac{1}{n!}P_{n}^{\left( -\lambda -1,-1\right)
}\left( x\right) .$\newline
Propositions \ref{P1}, \ref{P2} and \ref{P3} give%
\begin{eqnarray*}
L_{n}^{\left( \lambda \right) }\left( x\right) &=&\frac{e^{-x}}{n!}\underset{%
k\geq 0}{\sum }\left \langle \lambda +1+k\right \rangle _{n}\frac{x^{k}}{k!},
\\
L_{n}^{\left( \lambda \right) }\left( x\right) &=&\frac{1}{n!}\underset{k=0}{%
\overset{n}{\sum }}S_{-\lambda -1,-1}\left( n,k\right) x^{k}, \\
L_{n}^{\left( \lambda \right) }\left( x\right) &=&\frac{1}{n!}\underset{j=0}{%
\overset{n}{\sum }}\left( \underset{k=j}{\overset{n}{\sum }}\left \vert
s\left( n,k\right) \right \vert \left( \lambda +1\right) ^{k-j}\right) 
\mathcal{B}_{j}\left( x\right) , \\
L_{n}^{\left( \lambda \right) }\left( x\right) &=&\frac{1}{n!}\underset{k=0}{%
\overset{n}{\sum }}B_{n+1,k+1}^{\left( 1\right) }\left( \left \langle
1\right \rangle _{j};\left \langle \lambda +1\right \rangle _{j-1}\right)
x^{k}.
\end{eqnarray*}%
To write $P_{n}^{\left( \alpha ,\beta \right) }\left( x\right) $ in the
basis $\left \{ 1,L_{1}^{\left( \lambda \right) }\left( x\right) ,\ldots
,L_{n}^{\left( \lambda \right) }\left( x\right) \right \} ,$ set $\left(
\alpha ^{\prime },\beta ^{\prime }\right) =\left( -\lambda -1,-1\right) $ in
Proposition \ref{P4} to obtain%
\begin{equation*}
P_{n}^{\left( \alpha ,\beta \right) }\left( x\right) =\underset{j=0}{\overset%
{n}{\sum }}\left( -1\right) ^{j}j!S_{\alpha -\left( \lambda +1\right) \beta
,-\beta }\left( n,j\right) L_{j}^{\left( \lambda \right) }\left( x\right) .
\end{equation*}%
Theorem \ref{T3} proves the known property on the generalized Laguerre
polynomials $L_{n}^{\left( \lambda \right) },$ $n\geq 1,$ that have only
real zeros (here for $\lambda \geq -2$), for more information about the real
zeros of Laguerre polynomials see for example \cite{dim}. Theorem \ref{T4}
shows that, for $x>0,$ there hold%
\begin{eqnarray*}
L_{n}^{\left( \lambda \right) }\left( \frac{1}{x}\right) &=&\frac{\left(
-1\right) ^{n}}{n!}x^{n+1+\lambda }e^{-\frac{1}{x}}\left( \frac{d}{dx}%
\right) ^{n}\left( x^{-\lambda -1}e^{\frac{1}{x}}\right) , \\
L_{n}^{\left( \lambda \right) }\left( x\right) &=&\frac{x^{-\lambda }e^{-x}}{%
n!}\left( \frac{d}{dx}\right) ^{n}\left( x^{n+\lambda }e^{x}\right) .
\end{eqnarray*}%
We remark that for $\lambda =2r-1$ be a positive odd integer, we obtain%
\begin{equation*}
L_{n}^{\left( 2r-1\right) }\left( \frac{1}{x}\right) =\frac{\left( -1\right)
^{n}}{n!}x^{n+2r}e^{-\frac{1}{x}}\left( \frac{d}{dx}\right) ^{n}\left( \frac{%
1}{x^{2r}}e^{\frac{1}{x}}\right) =\frac{1}{n!}\underset{k=0}{\overset{n}{%
\sum }}\frac{L_{r}\left( n+r,k+r\right) }{x^{k}},
\end{equation*}%
where $L_{r}\left( n,k\right) $ is the $\left( n,k\right) $-th $r$-Lah
number, see \cite{boy,mih5,nyu}.

\subsection{Application to the associated Lah polynomials}

Let $m$ be a positive integer. The sequence of the associated Lah
polynomials $\left( \mathcal{L}_{n}^{\left( m\right) }\left( x\right)
\right) $ are studied in \cite{ahu,nan} and are defined by%
\begin{equation*}
\underset{n\geq 0}{\sum }\mathcal{L}_{n}^{\left( m\right) }\left( x\right) 
\frac{t^{n}}{n!}=\exp \left( x\left( \left( 1-t\right) ^{-m}-1\right)
\right) .
\end{equation*}%
This shows that $\mathcal{L}_{n}^{\left( m\right) }\left( x\right)
=P_{n}^{\left( 0,-m\right) }\left( x\right) .$ Propositions \ref{P1}, \ref%
{P2} and \ref{P3} give%
\begin{equation*}
\begin{array}{ll}
\mathcal{L}_{n}^{\left( m\right) }\left( x\right) =e^{-x}\underset{k\geq 0}{%
\sum }\left \langle mk\right \rangle _{n}\frac{x^{k}}{k!}, & \mathcal{L}%
_{n}^{\left( m\right) }\left( x\right) =\underset{k=0}{\overset{n}{\sum }}%
S_{0,-m}\left( n,k\right) x^{k}, \\ 
\mathcal{L}_{n}^{\left( m\right) }\left( x\right) =\underset{j=0}{\overset{n}%
{\sum }}m^{j}\left \vert s\left( n,j\right) \right \vert \mathcal{B}%
_{j}\left( x\right) , & \mathcal{L}_{n}^{\left( m\right) }\left( x\right) =%
\underset{k=0}{\overset{n}{\sum }}B_{n,k}\left( \left \langle m\right
\rangle _{j}\right) x^{k}.%
\end{array}%
\end{equation*}%
To write $P_{n}^{\left( \alpha ,\beta \right) }\left( x\right) $ in the
basis $\left \{ 1,\mathcal{L}_{1}^{\left( m\right) }\left( x\right) ,\ldots ,%
\mathcal{L}_{n}^{\left( m\right) }\left( x\right) \right \} ,$ set $\left(
\alpha ^{\prime },\beta ^{\prime }\right) =\left( 0,-m\right) $ in
Proposition \ref{P4} to obtain 
\begin{equation*}
P_{n}^{\left( \alpha ,\beta \right) }\left( x\right) =\underset{j=0}{\overset%
{n}{\sum }}\left( -1\right) ^{j}S_{\alpha ,-\frac{\beta }{m}}\left(
n,j\right) \mathcal{L}_{j}^{\left( m\right) }\left( x\right) .
\end{equation*}%
Theorem \ref{T3} proves a known property of the associated Lah polynomials $%
\mathcal{L}_{n}^{\left( m\right) },$ $n\geq 1,$ that have only real zeros
and Theorem \ref{T4} shows that, for $x>0,$ there hold%
\begin{eqnarray*}
\mathcal{L}_{n}^{\left( m\right) }\left( \frac{1}{x^{m}}\right) &=&\left(
-1\right) ^{n}x^{n}e^{-1/x^{m}}\left( \frac{d}{dx}\right) ^{n}\left(
e^{1/x^{m}}\right) , \\
\mathcal{L}_{n}^{\left( m\right) }\left( x^{m}\right) &=&xe^{-x^{m}}\left( 
\frac{d}{dx}\right) ^{n}\left( x^{n-1}e^{x^{m}}\right) .
\end{eqnarray*}

\end{document}